\documentclass[12pt]{amsart}    
\usepackage{amscd}

\def\cal#1{\mathcal{#1}}

\def\NZQ{\Bbb}               
\def\NN{{\NZQ N}}

\def\PP{{\NZQ P}}

\def\AA{{\NZQ A}}
\def\PP{{\NZQ P}}

\def\frk{\frak}               

\def\mm{{\frk m}}

\def\opn#1#2{\def#1{\operatorname{#2}}} 
\opn\chara{char}
\opn\length{\ell}
\opn\pd{pd}
\opn\rk{rk}
\opn\projdim{proj\,dim}
\opn\rank{rank}
\opn\depth{depth}
\opn\grade{grade}
\opn\height{ht}
\opn\embdim{emb\,dim}
\opn\codim{codim}
\def\OO{\mathcal{O}}
\opn\Tr{Tr}
\opn\bigrank{big\,rank}
\opn\superheight{superheight}\opn\lcm{lcm}
\opn\trdeg{tr\,deg}%
\opn\reg{reg}
\opn\lreg{lreg}
\opn\div{div}
\opn\Div{Div}
\opn\WDiv{WDiv}
\opn\cl{cl}
\opn\Cl{Cl}
\opn\Spec{Spec}
\opn\Supp{Supp}
\opn\supp{supp}
\opn\Sing{Sing}
\opn\Ass{Ass}
\opn\Assh{Assh}
\opn\Min{Min}
\opn\Reg{Reg}
\opn\Ann{Ann}
\opn\Rad{Rad}
\opn\Soc{Soc}
\opn\Socle{Socle}
\opn\Ker{Ker}
\opn\Coker{Coker}
\opn\Im{Im}
\opn\Hom{Hom}
\opn\Mor{Mor}
\opn\Tor{Tor}
\opn\Ext{Ext}
\opn\End{End}
\opn\Aut{Aut}
\opn\id{id}

\opn\nat{nat}
\opn\pff{pf}
\opn\Pf{Pf}
\opn\GL{GL}
\opn\SL{SL}
\opn\mod{mod}
\opn\ord{ord}
\opn\Proj{Proj}
\opn\aff{aff}
\opn\con{conv}
\opn\relint{relint}
\opn\st{st}
\opn\lk{lk}
\opn\cn{cn}
\opn\core{core}
\opn\vol{vol}
\opn\link{link}
\opn\star{star}
\opn\gr{gr}

\def\pot#1#2{#1[\kern-0.28ex[#2]\kern-0.28ex]}

\opn\dirlim{\underrightarrow{\lim}}
\opn\inivlim{\underleftarrow{\lim}}
\let\iso=\cong

\let\to=\rightarrow

\let\To=\longrightarrow

\def\Implies{\ifmmode\Longrightarrow \else
     \unskip${}\Longrightarrow{}$\ignorespaces\fi}
\def\implies{\ifmmode\Rightarrow \else
     \unskip${}\Rightarrow{}$\ignorespaces\fi}
\def\iff{\ifmmode\Longleftrightarrow \else
     \unskip${}\Longleftrightarrow{}$\ignorespaces\fi}

\let\:=\colon
\opn\H{H}
\opn\Pic{Pic}

\newtheorem{Theorem}{Theorem}
\newtheorem{Corollary}[Theorem]{Corollary}

\newtheorem{Proposition}[Theorem]{Proposition}

\let\epsilon\varepsilon
\def\FF{{\NZQ F}}

\def\OO{{\cal O}} 

\opn\inii{in}
\opn\inim{inm}
\opn\set{set}
\def\pnt{{\raise0.5mm\hbox{\large\bf.}}}

\begin{document}

\title{Kodaira type vanishing theorem\\ for the Hirokado variety}

\author{Yukihide Takayama}
\address{Yukihide Takayama, Department of Mathematical
Sciences, Ritsumeikan University, 
1-1-1 Nojihigashi, Kusatsu, Shiga 525-8577, Japan}
\email{takayama@se.ritsumei.ac.jp}

\def\Coh#1#2{H_{\mm}^{#1}(#2)}
\def\eCoh#1#2#3{H_{#1}^{#2}(#3)}

\newcommand{\AppTh}{Theorem~\ref{approxtheorem} }
\def\da{\downarrow}
\newcommand{\ua}{\uparrow}
\newcommand{\namedto}[1]{\buildrel\mbox{$#1$}\over\rightarrow}
\newcommand{\bdel}{\bar\partial}
\newcommand{\proj}{{\rm proj.}}

\newenvironment{myremark}[1]{{\bf Note:\ } \dotfill\\ \it{#1}}{\\ \dotfill
{\bf Note end.}}
\newcommand{\transdeg}[2]{{\rm trans. deg}_{#1}(#2)}
\newcommand{\mSpec}[1]{{\rm m\hbox{-}Spec}(#1)}

\newcommand{\tbf}{{{\Large To Be Filled!!}}}

\pagestyle{plain}
\maketitle

\def\gCoh#1#2#3{H_{#1}^{#2}\left(#3\right)}
\def\subsetneq{\raisebox{.6ex}{{\small $\; \underset{\ne}{\subset}\; $}}}
\opn\Exc{Exc}

\def\HHom{{\cal Hom}}

\begin{abstract}
The Hirokado variety is a Calabi-Yau threefold in characteristic~$3$ 
that is not liftable either to characteristic~$0$ or 
the ring $W_2$ of the second Witt vectors.
Although Deligne-Illusie-Raynaud type Kodaira vanishing cannot be
applied, we show that  $H^1(X, L^{-1})=0$, for an ample line 
bundle such that $L^3$ has a non-trivial global section, holds for this variety. 
\end{abstract}

\section{Introduction}

Although  Calabi-Yau threefolds have the unobstructed deformation in characteristic~$0$
and every K3 surface in positive characteristic can be lifted to
characteristic~$0$,  
the situation is quite different for Calabi-Yau
threefolds in positive characteristic. Namely, some of
Calabi-Yau threefolds in characteristic $2$ or $3$ cannot be 
lifted to characteristic~$0$ as shown in \cite{Hi99, Schr03, Hi07, Hi08, Scho}.
It is also known that there are unliftable 3-dimensional 
Calabi-Yau algebraic spaces in many positive characteristics \cite{CS, CSC}.

On the other hand, it is well known that the liftability problem is 
closely related to Kodaira vanishing theorem.  Namely, 
$W_2(K)$-liftability, where $W_2(K)$ is the ring of the second Witt vectors,
for varieties $X$ over an algebraically closed field $K$ of 
$\chara(K)=p>0$  with $\dim X\leq p$ is a sufficient condition for Kodaira 
vanishing \cite{DI}. However it is not clear whether Kodaira  vanishing 
holds without $W_2(K)$-liftability.

As far as the author is aware, it is not known whether Kodaira vanishing
holds on any of non-liftable Calabi-Yau threefolds in positive characteristic.
In \cite{TakPAMS} the author studied possibility of constructing a
Calabi-Yau threefold as a counterexample to Kodaira vanishing and
showed that it is possible if there is a surface of general type with
certain condition.  However, we do not know if such a surface exists.
Ekedahl \cite{Eke04} proved that both the
Hirokado variety \cite{Hi99} and the Schr\"oer variety \cite{Schr03}
are not $W_2(K)$-liftable.  But this does not 
necessarily imply that Kodaira vanishing does not hold.

In this paper, we show that on the  Hirokado variety Kodaira
vanishing holds to some extent. Namely,  we have  $H^1(X, L^{-1})=0$ for any ample line bundle $L$ such that $L^3$ has a 
non-trivial global section (Theorem~\ref{main}). 
This means that there is an example that (a part of) Kodaira  vanishing 
holds even if it is not $W_2(K)$-liftable. 
The proof is a rather easy consequence of Ekedahl's 
interpretation of the Hirokado variety as a Deligne-Lusztig type 
variety associated to the Grassmannian $Gr(2,4)$ together 
with the theory of pre-Tango structure \cite{Tan,  Ray, Mu11, Takeda}.

In the next section, we will briefly review the Ekedahl's reconstruction 
\cite{Eke04} of the Hirokado variety.
We also summarize the theory of pre-Tango structure,
which plays 
an essential role for construction of counterexamples 
to Kodaira vanishing.  
Then the main theorem will be proved in section~3.

\section{preliminaries}
In the following,  $K$ denotes an algebraically closed field of $\chara(K)=p>0$.

\subsection{Hirokado variety}

Let $\AA^3_K := \Spec K[x,y,z]\subset\PP^3_K$ be an affine open subset.
Consider the derivation $\delta$ on $\AA^3_K$:
\begin{equation}
\label{derivation}
   \delta = (x^p-x)\frac{\partial}{\partial x}
          + (y^p-y)\frac{\partial}{\partial y}
          + (z^p-z)\frac{\partial}{\partial z}.
\end{equation}
$\delta$ determines a vector field on $\PP^3_K$, which we will also denote 
by $\delta$. $\delta$ is $p$-closed, i.e., $\delta^p = f \delta$ 
for some element $f$
of the function field $K(X)$.
Moreover, 
$\delta$ has  $p^3 + p^2 + p +1$ isolated singular points, namely 
$\delta=0$ on $\PP^3({\FF_p}) = \{[z_0:z_1:z_2:z_3]\;\vert\; z_i^p = z_i,\; i=0,1,2,3\}$.
Since these singular points are isolated they can be resolved by 
one point blow-ups. Now we obtain the following diagram.
\begin{equation}
\label{diagram}
\begin{CD}
S           @>{\psi}>>   X \\
@V{\pi}VV             @V{\tilde{\pi}}VV \\
\PP^3_K     @>{\psi'}>> V       \\
\end{CD}
\end{equation}
where $\pi : S\To \PP^3_K$ is the blow-up centered at $\PP^3(\FF_p)$,
$\psi'$ and $\psi$ are the quotient maps by $\delta$ and $\pi^*\delta$
respectively. Namely,  $\OO_V = \{a\in \OO_{\PP^3_K}\;:\; \delta(a)=0\}$
and $\OO_X = \{a\in \OO_{S}\;:\; \pi^*\delta(a)=0\}$).
$\tilde{\pi}$ is the naturally induced morphism.
This variety $X$ is Calabi-Yau only when $p=3$, and we have 
the following non-liftability:

\begin{Theorem}[Corollary~2.3~\cite{Hi99} and Theorem~A~\cite{Eke04}]
If $p=3$, Hirokado variety $X$ is Calabi-Yau threefold 
that cannot be lifted to characteristic~$0$ or  $W_2(K)$.
Namely, there is no smooth projective morphism
\begin{equation*}
\psi: {\frak X}\To \Spec R,
\end{equation*}
where $R$ is a discrete valuation ring of mixed characteristic or 
$R = W_2(K)$, whose special fiber is isomorphic to $X$. 
In particular, Deligne-Illusie-Raynaud type Kodaira vanishing 
does not hold on $X$.
\end{Theorem}

\subsection{Ekedahl's reconstruction of Hirokado variety}

\subsubsection{Foliation and non-standard Gauss map} 
For a smooth variety $X$ with $n=\dim X$ 
and  its tangent sheaf ${\cal T}_X$,
a subbundle ${\cal E}\subset {\cal T}_X$ of constant rank $r$
that is closed under Lie brackets and $p$-th powers
is called a {\em 1-foliation},  on $X$.  For simplicity, we will call it 
{\em foliation} in the following.
For a purely inseparable finite flat morphism 
$f : X\to Y$
of degree $p$, the kernel ${\cal E}:= \Ker df$
of the differential of $f$  is a foliation. 
In this case, $f$ is called 
the {\em quotient map} by ${\cal E}$ and 
we denote as $Y = X/{\cal E}$. Then we have 
$\OO_Y \iso  \{a\in\OO_X\;\vert\; \delta(a)=0\mbox{ for all }\delta\in{\cal E}\}$.
See, for example, \cite{Eke87, Eke88} for the detail about foliation.

Let ${\cal G}$ be the Grassmannian bundle of $r$-dimensional 
subspaces of ${\cal T}_X$.
By considering a coordinate neighborhood $U\; (\subset X)$ 
that trivializes ${\cal T}_X$ and thus ${\cal G}$,
we can define a kind of  Gauss
map, which is 
the composition of the section
\begin{equation*}
e : U\To {\cal G} \qquad\mbox{such that }
U\ni x\longmapsto {\cal E}_x \in {\cal G}_x
\end{equation*}
and the projection ${\cal G}\vert_U\To Gr(r,n):= \{V\subset K^n\;:\; \dim V= r\}$.

Now we consider the special case of $U=\AA^n_K \subset X := \PP^n_K$
together with a foliation ${\cal E}$ of rank~$r$ on $U$ that can be 
extended to $X$.
By pulling back with the natural map
$\varphi: \AA^{n+1}-\{0\} \to \PP^n$,
 we obtain the foliation
$\varphi^*{\cal E}$ on $\varphi^{-1}(U)$, which induces
the Gauss map in our sense
\begin{equation*}
g' : \varphi^{-1}(U)\To Gr(r+1, n+1). 
\end{equation*}
Since this map is invariant 
under  the ${\bf G}_m$-action, we obtain the new Gauss map
\begin{equation*}
    g: U \To Gr(r+1, n+1)
\end{equation*}
which can be extended to $X =\PP^n_K$.
Moreover, we have 
\begin{Proposition}[cf. Proposition~2.2~\cite{Eke04}]
\label{EK:prop2.2}
The Gauss map $g: \PP^n_K\to Gr(r+1,n+1)$ factors through the 
quotient map $\PP^n_K\to \PP^n_K/{\cal E}$ by ${\cal E}$. In particular,
there exists the induced Gauss map
\begin{equation*}
       g_{\cal E} : \PP_K^n/{\cal E}\To Gr(r+1,n+1)
\end{equation*}
such that $g$ is the composition of the quotient map and 
$g_{\cal E}$.
\end{Proposition}

\subsubsection{Hirokado foliation on $\PP^n_K$} 
Now we consider the case of $n=3$ and $r=1$. 
Let ${\cal E}$ be the foliation on $\AA^3_k$ and on $\PP^3_K$ 
generated by the derivation $\delta$ as in (\ref{derivation}).
Then we have 
\begin{equation*}
\varphi^*{\cal E}
 =\left\langle
    \sum_{i=1}^{4} X_i^p\frac{\partial}{\partial X_i},\;
    \sum_{i=1}^{4} X_i\frac{\partial}{\partial X_i}
  \right\rangle
\end{equation*}
where the natural map $\varphi: \AA^4-\{0\} \to \PP^3$ is 
defined by $\varphi((X_1,X_2,X_3,X_4)) = [X_1:X_2:X_3:X_4]$.
Then by Proposition~\ref{EK:prop2.2} we obtain 
the Gauss map $g':\PP^3_K\backslash \PP^3_K(\FF_p) \to Gr(2,4)$
and also $g_{{\cal E}'}: V\to Gr(2,4)$ such that 
$g' = g_{{\cal E}'}\circ \psi'$.
Moreover, these Gauss maps can be extended to the blown-up
varieties and  we have the Gauss maps
\begin{equation*}
\hat{g}: S\To Gr(2,4)\quad\mbox{and}\quad
g : X\To Gr(2,4)
\end{equation*}
such that $\hat{g} = g\circ\psi$. 
(Lemma~2.3~\cite{Eke04}).

\subsubsection{the Hirokado variety as a Deligne-Lusztig type variety}
Now we consider a subvariety ${\cal F}$ of the Grassmannian $Gr(2,4)$
defined as 
\begin{equation*}
{\cal F}  =  \{ W\subset K^{4}\;:\; \dim W =2,\; W\cap F^*W\ne\emptyset\}
\end{equation*}
and the flag variety
\begin{equation*}
\tilde{{\cal F}}
      =  \{ L\subset W \subset K^{4}\;:\; \dim L=1,\; \dim W=2,\; L\subset F^*W\}
\end{equation*}
where $F: K^{4} \To K^{4}$ is the Frobenius morphism. We have 
the forgetful morphism
\begin{equation*}
\Psi: \tilde{\cal F}\To {\cal F}
\qquad\mbox{such that } (L\subset W \subset K^{4})\longmapsto (W\subset K^{4}).
\end{equation*}
Then we have 
\begin{Proposition}[Proposition~2.4~\cite{Eke04}]
\label{eke:prop2.4}
${\cal F}$ contains the image of $V$ by the Gauss map $g_{{\cal E}'}$, 
$\tilde{{\cal F}}\iso X$ and  $\Psi$ is a desingularization of ${\cal F}$,
whose exceptional set has codimension~$2$.
\end{Proposition}

It is well known that $Gr(2,4)$ is the hypersurface in $\PP^5_K$
defined by the Pl\"ucker relation 
\begin{equation*}
    q(X) := X_{12}X_{34} - X_{12}X_{24} + X_{14}X_{23} 
\end{equation*}
where $X_{ij}$ ($1\leq i<j\leq 4$) are the  homogeneous coordinates of $\PP^5_K$.
On the other hand, we can regard $Gr(2,4)$ as the set of projective 
lines in $\PP^3_K$. Thus every element $W\in {\cal F}$ can be regarded 
as a projective line in $\PP^3_K$ that intersects with the projective 
line represented by $F^*W$. This relation 
is described  by the bilinearization of the Pl\"ucker relation
\begin{equation*}
b(X, X^p) = X_{12}X^p_{34} - X_{12}X^p_{24} + X_{14}X^p_{23} 
+ X_{34}X^p_{12} 
- X_{24}X^p_{13} 
+ X_{23}X^p_{14},
\end{equation*}
which is a $(p+1)$-form (see, for example,  Proposition~12.1.1 \cite{ems}).
Thus ${\cal F}$ is 
a complete intersection of type $(2,p+1)$ in $\PP^5_K$.

Consequently, we have the following.
\begin{Theorem}[T.~Ekedahl \cite{Eke04}]
\label{ekedahl:interpretation}
The Hirokado variety $X$ is 
obtained from a complete intersection of type $(2,4)$ 
in $\PP^5_K$ by blowing up centered at the $40$ points 
in $\PP^5_{\FF_3}$.
\end{Theorem}

Now we consider the Hodge cohomologies of the Hirokado variety.

\begin{Proposition}[cf. Theorem~3.6(i) and Proposition~3.1(iii)~\cite{Eke04}]
\label{ekedahl:thm3.6(i)}
$\Psi_*\Omega_{\tilde{\cal F}}^1 = \Omega_{\cal F}^1$.
\end{Proposition}
\begin{proof}
Let $j : U\hookrightarrow {\cal F}$ be the inclusion of the non-singular locus.
Since $\Psi:\tilde{\cal F}\To {\cal F}$ is a small resolution
and its exceptional set  has codimension $2$,
we have $\Psi_*\Omega_{\tilde{\cal F}}^1 = j_*\Omega_U^i$.
Moreover, since ${\cal F}-U$ has codimension $\geq 2$,
we can define $\Omega_{{\cal F}}^1$ as the natural extension of 
$j_*\Omega_U^i$.  Thus we have 
$\Psi_*\Omega_{\tilde{\cal F}}^1 = \Omega_{\cal F}^1$.
\end{proof}

The cohomologies of smooth complete intersections have been
computed in Proposition~1.3 of \cite{De73}. 
This result has been  extended to the case of singular 
complete intersection by Ekedahl (Proposition~1.2~\cite{Eke04}),
from which we obtain the following.

\begin{Proposition}[Corollary~1.3~\cite{Eke04}]
\label{ekedahl:cor1.3}
Let $Y$ be an $n$-dimensional complete intersection in $\PP^r_K$ 
with only isolated singularities. For the Hodge numbers 
$h^{ij}_Y := \dim H^j(Y, \Omega_Y^i)$ we have that 
$h^{ij} = \delta_{ij}$ when $i+j <n$.
\end{Proposition}

\begin{Corollary}
\label{ekedahl:prop3.5(i)}
$H^0({\cal F}, \Omega_{\cal F}^1)=0$.
\end{Corollary}

\subsection{Tango structure and Kodaira vanishing}

Let $X$ be a smooth projective variety. Then an ample divisor $D$, or
an ample line bundle $L= \OO_X(D)$, is called a pre-Tango structure
if there exists an element $\eta\in K(X)\backslash K(X)^p$
such that the K\"ahler
differential is $d\eta\in \Omega_X^1(-pD)$, which will be simply
denoted as $(dy)\geq pD$.

If there exists a pre-Tango structure $L=\OO_X(D)$, 
we have $H^1(X,L^{-1})\ne{0}$. In fact, consider 
the absolute Frobenius map $F:\OO_X(-D)\to \OO_X(-pD)$ and 
set $B_X(-D) :=\Coker F$. Then we have the exact sequence
\begin{equation*}
0\To H^0(X, B_X(-D)) \To H^1(X, \OO_X(-D)) \overset{F}{\To} H^1(X, \OO_X(-pD))
\end{equation*}
where we have $H^0(X, B_X(-D)) = \{df \in K(X)\;\vert\; (df)\geq pD\}$
and this is non-trivial since $D$ is  pre-Tango.
Notice that multiplication by a non-trivial element from $H^0(X, B_X(-D))$ 
gives an embedding $\OO_X(pD) \hookrightarrow \Omega_X^1$.
See, for example, \cite{Takeda, Tan, Ray, Mu11} for more detail information 
on the pre-Tango structure.

The inclusion $H^0(X, B_X(-D))\subset H^1(X, \OO_X(-D))$ may be strict, which 
means that, for an ample line bundle $L$, $H^1(X, L^{-1})\ne{0}$ does not 
always mean that $L$ is a pre-Tango structure. However, we have

\begin{Proposition}
\label{knvTango}
If $H^1(X, L^{-1})\ne{0}$ for an ample line bundle $L$, then 
$L^n$ is a pre-Tango structure for some integer $n\geq 1$.
\end{Proposition}
\begin{proof}
Since $X$ is normal, Enriques-Severi-Zariski's theorem shows that 
iterated Frobenius maps
\begin{equation*}
       F^e : H^1(X, L^{-1}) \To H^1(X, L^{-p^e})\qquad (e\gg 0)
\end{equation*}
are trivial. Thus we have $H^0(X, B_X(-nD)) = H^1(X, L^{-n})(\ne{0})$ 
for a sufficiently large $n\in \NN$, or  precisely for  $n=p^e$
such that  $H^1(X, L^{-p^e})\ne{0}$ but $H^1(X, L^{-p^{e+1}})=0$. 
\end{proof}

\section{Main theorem}

Now we can prove the main result.

\begin{Theorem}
\label{main}
Let $X$ be  the Hirokado variety and $L$ be an ample line bundle
with $H^0(X, L^3)\ne{0}$. 
Then we have $H^1(X, L^{-1})=0$.
\end{Theorem}

\begin{proof}
Assume that $H^1(X, L^{-1})\ne{0}$ for some ample line bundle
$L$. Then by Proposition~\ref{knvTango} there exists an integer
$n\geq 1$ such that ${\cal L} = L^n$ is a pre-Tango structure.
Thus we have  ${\cal L}^p \subset \Omega_X^1$, $p=3$

On the other hand, by Proposition~\ref{eke:prop2.4} and
Theorem~\ref{ekedahl:interpretation},
we have $X \iso \tilde{\cal F}$ together with a desingularization
$\Psi : \tilde{\cal F}\to {\cal F}$
of a $(2,4)$-type complete intersection ${\cal F}$ in $\PP^5_K$.
Then we compute 
\begin{equation*}
H^0(\tilde{\cal F}, \Omega_{\tilde{\cal F}}^1)
= H^0({\cal F}, \Psi_*\Omega_{\tilde{\cal F}}^1) 
= H^0({\cal F}, \Omega_{\cal F}^1)
= 0 
\end{equation*}
by 
Proposition~\ref{ekedahl:thm3.6(i)}
and Corollary~\ref{ekedahl:prop3.5(i)}.
Since $H^0(\tilde{\cal F}, {\cal L}^p) \subset H^0(\tilde{\cal F},
\Omega_{\tilde{\cal F}}^1)$, we then have $H^0(X, L^{np}) 
= H^0(\tilde{\cal F}, {\cal L}^p)=0$,
which contradicts the assumption $H^0(X, L^p)\ne{0}$, $p=3$.
\end{proof}




\begin{thebibliography}{99}


\bibitem{Scho} 
  C.~Schoen, Desingularized fiber products of semi-stable
  elliptic surfaces with vanishing third Betti
  number. Compos. Math. 145 (2009), no. 1, 89--111.

\bibitem{CSC}  
S.~Cynk and M.~Sch\"utt, Non-liftable Calabi-Yau spaces,
arXiv:0910.2349

\bibitem{CS} 
 S.~Cynk and D.\ van Straten, Duco. Small resolutions and
non-liftable Calabi-Yau threefolds. Manuscripta Math. 130 (2009),
no. 2, 233--249.

\bibitem{ems} 
 M.~C.~Beltrametti, E.~Carletti, D.~Gallarati
and G.~M.~Bragadin, {\em Lectures on Curves, Surfaces and Projective Varieties,
A Classical View of Algebraic Geometry}, Textbook in Mathematics,,
European Mathematical Society, 2009.

\bibitem{De73} 
 P.~Deligne, Cohomologie des intersections completes, Groupes de monodromie en 
geometrie algebrique. II (Seminaire de Geometrie Algebrique du BoisMarie
1967--1969 (SGA 7 II), SLN. no.~340, Springer-Verlag, Berlin, 1973, pp.~39--61.


\bibitem{DI} 
 P.\ Deligne and L.\ Illusie, Luc,
Relevements modulo $p\sp 2$ et decomposition du complexe de de
Rham. Invent. Math. 89 (1987), no. 2, 247--270. 

\bibitem{Eke87} T.~Ekedahl, Foliations and inseparable
  morphisms. Algebraic geometry, Bowdoin, 1985 (Brunswick, Maine,
  1985), 139--149, Proc. Sympos. Pure Math., 46, Part 2,
  Amer. Math. Soc., Providence, RI, 1987.

\bibitem{Eke88} T.~Ekedahl, Canonical models of surfaces of general
  type in positive characteristic. Inst. Hautes Etudes
  Sci. Publ. Math. No. 67 (1988), 97--144.

\bibitem{Eke04} 
 T.\ Ekedahl, On non-liftable Calabi-Yau threefolds
(preprint), math.AG/0306435


\bibitem{Hi99} 
 M.\ Hirokado, A non-liftable Calabi-Yau threefold in
characteristic $3$. Tohoku Math. J. (2) 51 (1999), no. 4, 479--487.

\bibitem{Hi07} 
 M.\ Hirokado, H.\ Ito and N.\ Saito.
Calabi-Yau threefolds arising from fiber products of rational quasi-elliptic 
surfaces~I. 
Ark. Mat. 45 (2007), no. 2, 279--296.

\bibitem{Hi08} 
 M.\ Hirokado, H.\ Ito and N.\ Saito.
Calabi-Yau threefolds arising from fiber products of rational 
quasi-elliptic surfaces~II. 
Manuscripta Math. 125 (2008), no. 3, 325--343


\bibitem{Mu11} 
 S.\ Mukai, Counterexamples of Kodaira's vanishing and Yau's
inequality in characteristics, RIMP-1736, Research Institute of 
Mathematical Sciences, Kyoto University, 2011.


\bibitem{Schr03} 
 S.\ Schr\"oer, 
Some Calabi-Yau threefolds with obstructed deformations over the Witt
vectors. Compos. Math. 140 (2004), no. 6, 1579--1592.



\bibitem{TakPAMS} Y.~Takayama, 
Raynaud-Mukai construction and Calabi-Yau threefolds in positive characteristic,
to appear in Proc.~Amer.~Math.~Soc.

\bibitem{Takeda} 
 Y.~Takeda, Pre-Tango structures and uniruled
  varieties. Colloq. Math. 108 (2007), no. 2, 193--216


\bibitem{Tan} 
H.\ Tango, On the behavior of extensions of vector bundles
under the Frobenius map, Nagoya Math.\ J., 48, 73--89, 1972.

\bibitem{Ray} 
 M.\ Raynaud,  Contre-exemple au ``vanishing theorem''
  en caract\'eristique $p>0$. (French) C. P. Ramanujam---a tribute,
  pp. 273--278, Tata Inst. Fund. Res. Studies in Math., 8, Springer,
  Berlin-New York, 1978.


\end{thebibliography}
\end{document}